\documentclass[12pt,a4paper]{elsarticle}

\usepackage{amsmath,amssymb,amsthm}
\usepackage{mathtools}
\usepackage{graphicx}
\usepackage{hyperref}
\usepackage{booktabs}
\usepackage{array}
\usepackage{enumerate}
\usepackage{microtype}
\usepackage{xcolor}
\usepackage{hyperref}
\usepackage{academicons} 
\definecolor{orcidlogocol}{HTML}{A6CE39}
\newtheorem{theorem}{Theorem}

\newtheorem{corollary}[theorem]{Corollary}
\theoremstyle{definition}
\newtheorem{remark}[theorem]{Remark}

\newcommand{\R}{\mathbb{R}}
\newcommand{\norm}[1]{\left\|#1\right\|}
\newcommand{\ip}[2]{\langle #1,\,#2 \rangle}
\newcommand{\diag}{\operatorname{diag}}
\newcommand{\tr}{\operatorname{tr}}
\newcommand{\e}{\mathbf{e}}
\newcommand{\bx}{\mathbf{x}}
\newcommand{\ba}{\mathbf{a}}
\newcommand{\bb}{\mathbf{b}}
\newcommand{\bxi}{\boldsymbol{\xi}}
\DeclareMathOperator{\sprad}{\varrho}

\journal{ Bulletin of the Australian Mathematical Society}

\begin{document}

\begin{frontmatter}

\title{Sharp Convergence Rates and Optimal Weights \\ for Cimmino's Reflection Algorithm}

\author[b]{%
    Hemant Sharma\textsuperscript{\href{https://orcid.org/0009-0006-9020-3785}{\includegraphics[scale=0.04]{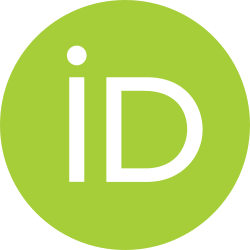}}}\corref{cor1}%
}
\ead{sharmahemant39@gmail.com}

\address[b]{Tungal School of Basic and Applied Sciences, Jamkhandi, Karnataka 587301, India.}
\cortext[cor1]{Corresponding author.}

\begin{abstract}
In this paper, Cimmino's classical reflection algorithm for solving the
$n\times n$ nonsingular linear system $A\bx=\bb$ is analysed through
the lens of spectral theory.
Reformulating the weighted iteration as $\e^{(\nu+1)}=M_w\,\e^{(\nu)}$,
where $M_w = I - A^\top D_w A$, the error is shown to contract by
the spectral radius $\sprad(M_w)$ at every step, with a sharp,
asymptotically tight bound.
For $n=2$, a closed-form expression for the contraction factor is derived,
\[
  \sprad(M_w) \;=\; |1-\mu|
    + \tfrac{1}{2}\sqrt{(w_1-w_2)^2 + 4w_1w_2\cos^2\!\theta},
\]
where $\mu=(w_1+w_2)/2$ and $\theta$ denotes the angle between the
hyperplane normals.
A central result of this paper is that the standard unit weights
$w_1^*=w_2^*=1$ are \emph{globally optimal} over all positive weight
pairs, uniquely achieving the minimum contraction factor
$\sprad^*=|\cos\theta|$ --- a quantity determined solely by the geometry
of the hyperplane normals.
The inter-normal angle $\theta$ thus emerges as the single diagnostic
parameter governing both convergence speed and weight selection.
Extensions to a single-step convergence criterion at $\theta=\pi/2$
and to an exact spectral rate for general~$n$ are also established.
\end{abstract}


\end{frontmatter}

\section{Introduction}

Cimmino's reflection algorithm~\cite{Cimmino1938}, published in 1938
and recently re-examined in~\cite{Guida2023}, solves the nonsingular
linear system
\begin{equation}\label{eq:system}
  A\bx = \bb, \qquad A \in \R^{n\times n},\quad \bb \in \R^n,
  \quad \det A \neq 0,
\end{equation}
by an iterative geometric procedure: starting from any $\bx^{(0)}\in\R^n$,
the algorithm reflects $\bx^{(0)}$ across each of the $n$ affine hyperplanes
$H_i = \{\bx:\ip{\ba_i}{\bx}=b_i\}$ (where $\ba_i^\top$ is the $i$-th
row of~$A$) and sets $\bx^{(1)}$ as the weighted centroid of the
reflected points.
Iterating yields a sequence $(\bx^{(\nu)})$ converging to the unique
solution~$\bxi$.
In matrix form~\cite{Guida2023},
\begin{equation}\label{eq:Cimmino_matrix}
  \bx^{(\nu+1)} = \bx^{(\nu)}
  + A^\top D\bigl(\bb - A\bx^{(\nu)}\bigr),
  \quad D = \diag\!\bigl(\norm{\ba_1}^{-2},\ldots,\norm{\ba_n}^{-2}\bigr).
\end{equation}
The method has been applied extensively in image reconstruction for
radiation therapy~\cite{Censor1988} and benefits from block and
parallel implementations~\cite{Aharoni1989,Torun2018}, making it
attractive for large sparse systems where sequential alternatives
are impractical.
While convergence is guaranteed whenever $A$ is nonsingular,
the analysis in~\cite{Guida2023} provides only the non-expansive bound
$\norm{\bx^{(\nu+1)}-\bxi}\le\norm{\bx^{(\nu)}-\bxi}$,
leaving two fundamental questions open:
\begin{enumerate}[(i)]
\item What is the \emph{exact} contraction rate at each iteration?
\item How should one choose the weights $w_i>0$ assigned to the reflected
      points to minimise that rate?
\end{enumerate}
We answer both questions completely for $n=2$
(Theorems~\ref{thm:rate} and~\ref{thm:optimal})
and provide an exact spectral rate for general~$n$
(Theorem~\ref{thm:general-n}).
The main finding is that \emph{unit weights are globally optimal},
achieving $\sprad^*=|\cos\theta|$, where $\theta$ is the sole geometric
parameter governing convergence.
This yields the first rigorous justification for the standard weight
choice used in practice.

The related Kaczmarz method~\cite{Kaczmarz1937} performs sequential
(not simultaneous) projections and tends to converge faster per step,
but is less suited to parallel implementation~\cite{Benzi2022}.
Randomised~\cite{Strohmer2009} and block~\cite{Needell2014} variants
have received substantial recent attention in the
Kaczmarz setting; the weight-optimality result obtained here has no
counterpart there.

\section{Problem Formulation and Notation}\label{sec:notation}

Let $P_i = \hat{\ba}_i\hat{\ba}_i^\top$ be the rank-one orthogonal
projection onto $\ba_i$, where $\hat{\ba}_i=\ba_i/\norm{\ba_i}$.
The \emph{weighted} Cimmino iteration assigns positive masses $w_i>0$
to each reflected point:
\begin{equation}\label{eq:weighted_iter}
  \bx^{(\nu+1)}
  = \bx^{(\nu)}
    + \sum_{i=1}^{n}
      \frac{w_i\bigl(b_i-\ip{\ba_i}{\bx^{(\nu)}}\bigr)}{\norm{\ba_i}^2}
      \,\ba_i
  = \bx^{(\nu)} + A^\top D_w\bigl(\bb - A\bx^{(\nu)}\bigr),
\end{equation}
where $D_w=\diag(w_1\norm{\ba_1}^{-2},\ldots,w_n\norm{\ba_n}^{-2})$.
Setting $\e^{(\nu)}=\bx^{(\nu)}-\bxi$ and using $A\bxi=\bb$, the
error satisfies the linear recursion
\begin{equation}\label{eq:error_recursion}
  \e^{(\nu+1)} = M_w\,\e^{(\nu)}, \qquad
  M_w := I - A^\top D_w A = I - \textstyle\sum_{i=1}^{n} w_i P_i.
\end{equation}
For $n=2$ we define the \emph{inter-normal angle}
$\theta\in(0,\pi)$ by $\cos\theta=\ip{\hat{\ba}_1}{\hat{\ba}_2}$.

Figure~\ref{fig:reflection} visualises the recursion~\eqref{eq:error_recursion}
geometrically for $n=2$.
The key observation, which underpins every theorem that follows, is that
\emph{all reflected points $Q^{(i)}$ lie on the sphere of radius
$\norm{\e^{(0)}}$ centred at the solution $Z$}; the weighted centroid
$P^{(1)}$ therefore lies strictly inside that sphere.
The angle $\theta$ labelled in the figure is the single quantity that
controls how fast the centroid moves toward $Z$: a wider opening between
the two hyperplanes (smaller $|\cos\theta|$) corresponds to faster
convergence, and at $\theta=\pi/2$ the centroid lands exactly at $Z$ in
one step (Corollary~\ref{cor:one_step}).

\begin{figure}[ht]
  \centering
  \includegraphics[width=0.72\textwidth]{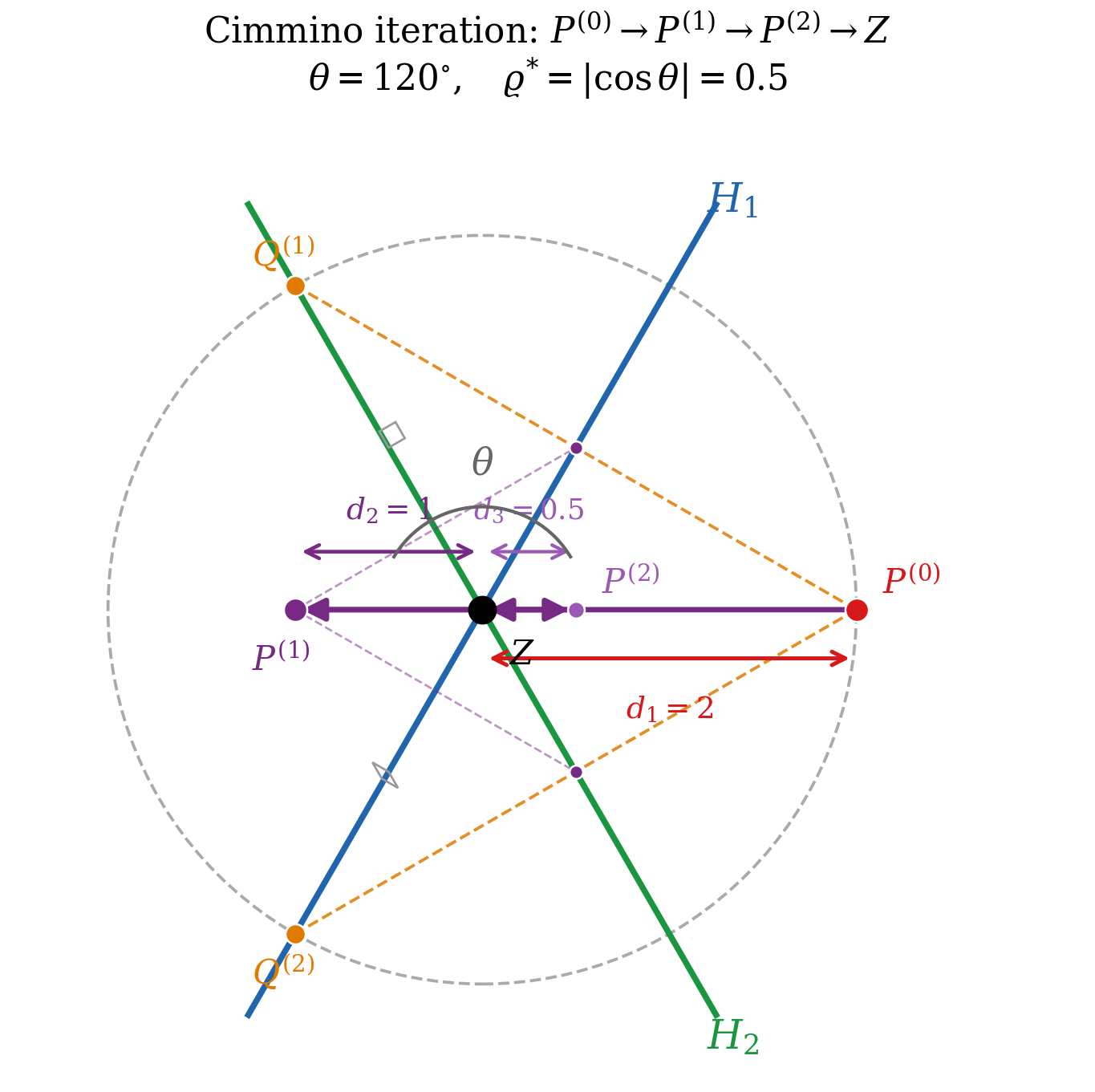}
  \caption{%
    \textbf{Geometric foundation (Section~\ref{sec:notation}).}
    Cimmino's method for $n=2$, with solution $Z$ at the origin and
    $P^{(0)}=(2,0)$.
    The dashed circle has radius $d_1=\norm{P^{(0)}-Z}=2$.
    Reflections $Q^{(1)},Q^{(2)}$ of $P^{(0)}$ across $H_1$ (blue)
    and $H_2$ (green) lie exactly on this circle --- a classical
    result~\cite[Lemma~2.2]{Guida2023}.
    The unit-weight centroid $P^{(1)}$ falls strictly inside, giving
    $d_2=1$; a second iteration from $P^{(1)}$ yields $P^{(2)}$ with
    $d_3=0.5$.
    Each iteration halves the error because
    $\varrho^*=|\cos\theta|=0.5$ at $\theta=120^{\circ}$.
  }
  \label{fig:reflection}
\end{figure}

The inter-normal angle $\theta$ is the sole quantity determining the
convergence rate (Theorem~\ref{thm:optimal}); no other information
about $A$ or $\bb$ is needed.

\section{Main Results}\label{sec:main}


\begin{theorem}[Spectral Convergence Criterion]\label{thm:convergence}
The weighted Cimmino iteration~\eqref{eq:weighted_iter} converges for
every $\bx^{(0)}\in\R^n$ if and only if $\sprad(M_w)<1$, equivalently,
all eigenvalues of $A^\top D_w A$ lie in $(0,2)$.
Moreover, for every $\nu\ge 0$,
\begin{equation}\label{eq:error_bound}
  \norm{\e^{(\nu)}}_2 \;\le\; \sprad(M_w)^{\nu}\,\norm{\e^{(0)}}_2,
\end{equation}
and the bound is asymptotically sharp.
\end{theorem}

\begin{proof}
From~\eqref{eq:error_recursion}, $\e^{(\nu)}=M_w^\nu\e^{(0)}$.
Since $B_w:=A^\top D_w A=\sum_i w_i P_i$ is symmetric positive definite
($A$ nonsingular, $w_i>0$), its eigenvalues
$0<\lambda_1\le\cdots\le\lambda_n$ are real and positive.
The eigenvalues of $M_w=I-B_w$ are $\mu_i=1-\lambda_i$, and
$\norm{M_w^\nu}_2=\sprad(M_w)^\nu$ for symmetric matrices,
giving the bound.
Convergence requires $|\mu_i|<1$ for all $i$, i.e.\
$\lambda_i\in(0,2)$.
Sharpness follows by choosing $\e^{(0)}$ in the eigenspace of the
dominant eigenvalue.
\end{proof}

Theorem~\ref{thm:convergence} replaces the bare inequality of~\cite{Guida2023}
with a quantitative envelope: the error at step $\nu$ is at most
$\sprad(M_w)^\nu$ times the initial error.
Figure~\ref{fig:envelopes} shows this envelope for three weight
choices at the same angle $\theta=120^{\circ}$.
The gap between the green (optimal) and red (poor) curves --- a factor
of $0.90^{12}/0.50^{12}\approx 180$ in error after twelve iterations
--- demonstrates that a na\"ive weight choice can be dramatically
suboptimal even when all weights are equal.
The figure also confirms that the bound~\eqref{eq:error_bound} is
tight: the optimal curve drops at exactly the rate predicted by
$\varrho^*=0.50$.

\begin{figure}[ht]
  \centering
  \includegraphics[width=0.82\textwidth]{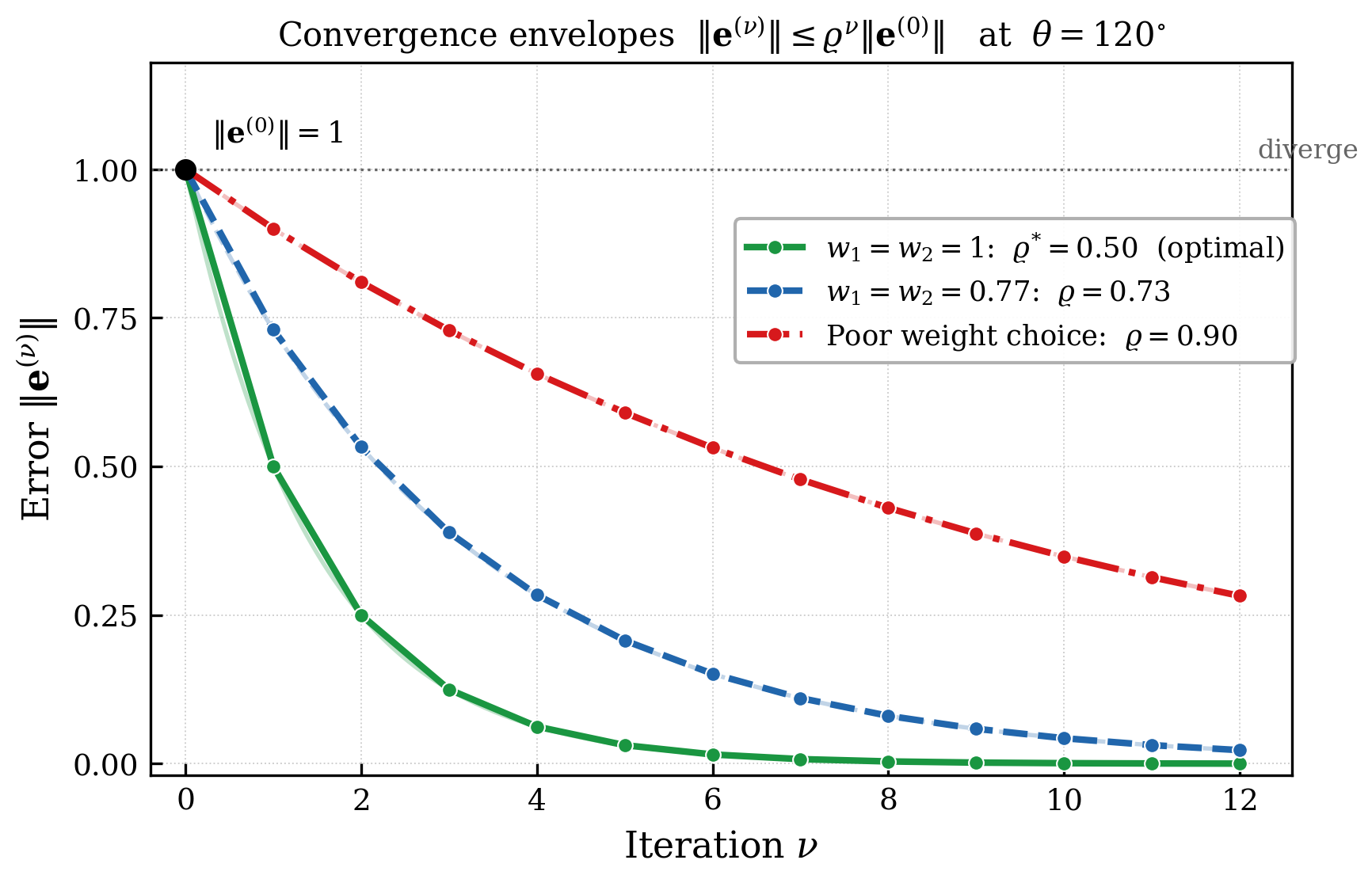}
  \caption{%
    \textbf{Error envelopes --- Theorem~\ref{thm:convergence}.}
    Three weight choices at $\theta=120^{\circ}$,
    all with $\norm{\e^{(0)}}=1$.
    Solid circles mark the actual iterate errors
    $\norm{\e^{(\nu)}}$ computed from~\eqref{eq:error_recursion};
    the dashed background curves are the continuous envelopes
    $\varrho^\nu$.
  }
  \label{fig:envelopes}
\end{figure}

The exact value of $\sprad(M_w)$ for $n=2$ is given next.


\begin{theorem}[Explicit Contraction Factor, $n=2$]\label{thm:rate}
Let $n=2$, $w_1,w_2>0$, $\mu=(w_1+w_2)/2$, and let $\theta$ be the
angle between the rows $\ba_1$ and $\ba_2$ of $A$.  Then
\begin{equation}\label{eq:rho_formula}
  \sprad(M_w) \;=\; |1-\mu|
    \;+\; \tfrac{1}{2}\sqrt{(w_1-w_2)^2 + 4w_1w_2\cos^2\!\theta}.
\end{equation}
\end{theorem}

\begin{proof}
The $2\times 2$ symmetric matrix $A^\top D_w A=w_1P_1+w_2P_2$
satisfies
\begin{align}
  \tr(w_1P_1+w_2P_2)  &= w_1+w_2 = 2\mu,           \label{eq:trace}\\
  \det(w_1P_1+w_2P_2) &= w_1w_2\sin^2\!\theta,      \label{eq:det}
\end{align}
where~\eqref{eq:det} uses
$\det(\alpha uu^\top+\beta vv^\top)=\alpha\beta(1-\ip{u}{v}^2)$ for
unit vectors.
Its eigenvalues are
$\lambda_\pm=\mu\pm\tfrac{1}{2}\Delta$,
$\Delta=\sqrt{(w_1-w_2)^2+4w_1w_2\cos^2\theta}$.
The eigenvalues of $M_w$ are $(1-\mu)\mp\Delta/2$, giving
$\sprad(M_w)=\max(|(1-\mu)-\Delta/2|,\,|(1-\mu)+\Delta/2|)
=|1-\mu|+\Delta/2$.
\end{proof}


\begin{theorem}[Global Optimality of Unit Weights]\label{thm:optimal}
For $n=2$ and all $w_1,w_2>0$,
\begin{equation}\label{eq:lower_bound}
  \sprad(M_w) \;\ge\; |\cos\theta|,
\end{equation}
with equality if and only if $w_1=w_2=1$.
Consequently $\varrho^*:=\min_{w_1,w_2>0}\sprad(M_w)=|\cos\theta|$,
achieved uniquely at the standard unit weights.
\end{theorem}

\begin{proof}
Write $\mu=(w_1+w_2)/2$ and $s=(w_1-w_2)/2$, so $w_1w_2=\mu^2-s^2$.
Formula~\eqref{eq:rho_formula} gives
\begin{equation}\label{eq:rho_mus}
  \sprad(M_w)
  = |1-\mu| + \sqrt{s^2\sin^2\!\theta + \mu^2\cos^2\!\theta}
  \ge |1-\mu| + \mu|\cos\theta| =: g(\mu),
\end{equation}
since $s^2\sin^2\theta\ge 0$ and $\mu>0$.
For $0<\mu\le 1$:
$g(\mu)=1-\mu(1-|\cos\theta|)\ge g(1)=|\cos\theta|$.
For $\mu\ge 1$:
$g(\mu)=\mu(1+|\cos\theta|)-1\ge g(1)=|\cos\theta|$.
Hence $\sprad(M_w)\ge|\cos\theta|$ for all $\mu>0$.
Equality in the first inequality requires $s=0$
(i.e.\ $w_1=w_2$, since $\sin\theta>0$ on $(0,\pi)$) and in the
second requires $\mu=1$
(i.e.\ $w_1+w_2=2$); together, $w_1=w_2=1$ is the unique minimiser.
\end{proof}

Figure~\ref{fig:rho_theta} translates Theorem~\ref{thm:optimal} into
a visual statement: for \emph{every} angle $\theta\in(0^{\circ},180^{\circ})$,
the green curve (unit weights, $\varrho^*=|\cos\theta|$) lies at or below
every other curve.
Three further observations are worth noting.
First, all four curves share the same value only at the two endpoints
$\theta\to 0^{\circ}$ and $\theta\to 180^{\circ}$, where all weights
converge to $\varrho=1$ --- meaning convergence becomes arbitrarily
slow as the two hyperplanes become nearly parallel, regardless of
weight choice.
Second, the red curve ($w_1=w_2=1.4$) crosses the divergence boundary
$\varrho=1$ near the endpoints, showing that even equal weights can
cause divergence when they differ from unity.
Third, Corollary~\ref{cor:one_step} is visible as the green curve
touching zero exactly at $\theta=90^{\circ}$: when the hyperplane
normals are orthogonal, unit weights deliver exact convergence in a
single step.

\begin{figure}[ht]
  \centering
  \includegraphics[width=0.84\textwidth]{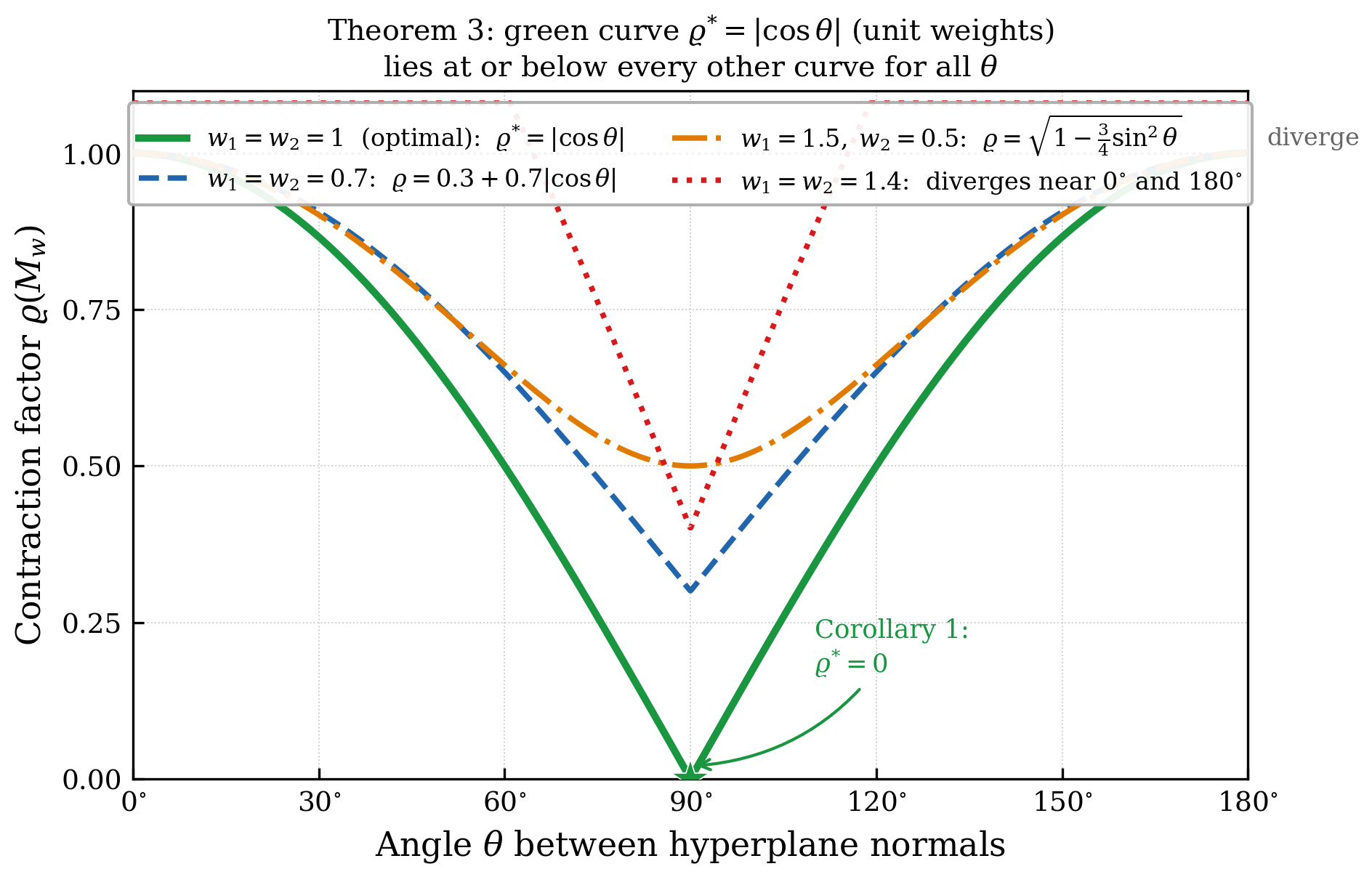}
  \caption{%
    \textbf{Visual proof of global optimality --- Theorem~\ref{thm:optimal}.}
    Contraction factor $\varrho(M_w)$ versus inter-normal angle $\theta$
    for four weight choices (Theorem~\ref{thm:rate}).
    The green solid curve $\varrho^*=|\cos\theta|$ (unit weights) lies
    at or below every other curve for all $\theta\in(0^{\circ},180^{\circ})$,
    and the red dotted curve ($w_1=w_2=1.4$) crosses the divergence
    boundary $\varrho=1$ near the endpoints.
  }
  \label{fig:rho_theta}
\end{figure}

The figure constitutes a graphical proof of
Theorem~\ref{thm:optimal}: the green curve is visibly and
provably the global minimum over all $w_1,w_2>0$ for every
$\theta$.
Crucially, the minimum value $\varrho^*=|\cos\theta|$ depends only
on the geometry of the hyperplane normals --- not on the right-hand
side $\bb$, nor on the magnitudes of the rows of $A$.
This makes $\theta$ a computable, preprocessing-time diagnostic:
if $|\cos\theta|$ is close to $1$ (nearly parallel hyperplanes),
Cimmino's method will converge slowly regardless of the weight choice,
and a different solver or preconditioner should be considered.

\begin{corollary}[Single-Step Convergence]\label{cor:one_step}
If $\theta=\pi/2$, then $\bx^{(1)}=\bxi$ for every
$\bx^{(0)}\in\R^2$.
\end{corollary}

\begin{proof}
When $\cos\theta=0$, $\{\hat{\ba}_1,\hat{\ba}_2\}$ is an orthonormal
basis of $\R^2$, so $P_1+P_2=I$ and $M=I-I=0$.
\end{proof}


\begin{theorem}[Exact rate and optimal scaling, general $n$]\label{thm:general-n}
Let $A\in\R^{n\times n}$ be nonsingular, $w_i>0$, and let
$B_w := A^{\top}D_w A = \sum_{i=1}^{n} w_i P_i$ have eigenvalues
$0<\lambda_1\le\cdots\le\lambda_n$.
Then
\begin{equation}\label{eq:exact-rho}
  \sprad(M_w) = \max\bigl(|1-\lambda_1|,\,|1-\lambda_n|\bigr),
\end{equation}
and the iteration converges if and only if $\lambda_n<2$.
Moreover, within the scaling family $\{\alpha w:\alpha>0\}$ the
optimal rate is
\begin{equation}\label{eq:optimal-scaling}
  \min_{\alpha>0}\sprad(M_{\alpha w})
  = \frac{\kappa(B_w)-1}{\kappa(B_w)+1},
  \qquad \kappa(B_w)=\lambda_n/\lambda_1,
\end{equation}
attained at $\alpha^{\star}=2/(\lambda_1+\lambda_n)$.
\end{theorem}

\begin{proof}
$B_w$ is symmetric positive definite, so its eigenvalues are positive
and $M_w=I-B_w$ is symmetric with eigenvalues $\{1-\lambda_i\}$.
Since $M_w$ is symmetric,
$\sprad(M_w)=\max_{1\le i\le n}|1-\lambda_i|$.
Each $\lambda_i\in[\lambda_1,\lambda_n]$ is a convex combination of
the endpoints, so convexity of $\lambda\mapsto|1-\lambda|$ gives
$|1-\lambda_i|\le\max(|1-\lambda_1|,|1-\lambda_n|)$; the reverse
inequality is trivial since $\lambda_1$ and $\lambda_n$ are among the
$\lambda_i$.
This proves~\eqref{eq:exact-rho}.
Convergence is equivalent to $|1-\lambda_i|<1$ for every $i$, i.e.\
$\lambda_n<2$ (the lower bound $\lambda_1>0$ is automatic).

For~\eqref{eq:optimal-scaling}, $B_{\alpha w}=\alpha B_w$ has
eigenvalues $\alpha\lambda_i$, so
$\sprad(M_{\alpha w})=\max(|1-\alpha\lambda_1|,|1-\alpha\lambda_n|)$,
a convex piecewise-linear function of $\alpha>0$.
Its two branches cross where
$1-\alpha\lambda_1=\alpha\lambda_n-1$, giving
$\alpha^{\star}=2/(\lambda_1+\lambda_n)$ and
$\sprad(M_{\alpha^{\star}w})
=1-\alpha^{\star}\lambda_1
=(\lambda_n-\lambda_1)/(\lambda_n+\lambda_1)
=(\kappa-1)/(\kappa+1)$.
\end{proof}

\begin{corollary}[Tight-frame one-step convergence]\label{cor:tight_frame}
If there exist $w_i^{\star}>0$ with
$\sum_{i=1}^{n} w_i^{\star} P_i = I$, then
$M_{w^{\star}}=0$ and $\bx^{(1)}=\bxi$ for every $\bx^{(0)}\in\R^n$.
At $n=2$ this reduces to Corollary~\ref{cor:one_step}
($\theta=\pi/2$).
\end{corollary}

\begin{proof}
Immediate from
$M_{w^{\star}}=I-\sum_i w_i^{\star}P_i=I-I=0$.
\end{proof}

\begin{remark}\label{rem:subsumes}
Theorem~\ref{thm:general-n} subsumes Theorem~\ref{thm:rate}: the
eigenvalues $\lambda_\pm=\mu\pm\Delta/2$ from
\eqref{eq:trace}--\eqref{eq:det} give
$\max(|1-\lambda_-|,|1-\lambda_+|)=|1-\mu|+\Delta/2$,
recovering~\eqref{eq:rho_formula}.
\end{remark}

\section{Numerical Illustration}\label{sec:numerics}

We illustrate Theorems~\ref{thm:rate}--\ref{thm:optimal} on two
concrete examples.

\medskip
\noindent\textbf{Example~1 (Suboptimal vs.\ optimal weights).}
Let
\[
  A = \begin{pmatrix}2&1\\1&2\end{pmatrix},\quad
  \bb = \begin{pmatrix}3\\3\end{pmatrix},\quad
  \bxi = \begin{pmatrix}1\\1\end{pmatrix}.
\]
Here $\norm{\ba_1}=\norm{\ba_2}=\sqrt{5}$ and
$\cos\theta=\ip{\ba_1}{\ba_2}/(\norm{\ba_1}\norm{\ba_2})=4/5$,
giving $\varrho^*=4/5$.
Direct computation confirms $A^\top DA=\tfrac{1}{5}
\bigl(\begin{smallmatrix}5&4\\4&5\end{smallmatrix}\bigr)$
with eigenvalues $9/5$ and $1/5$, so $\sprad(M)=4/5$.

\medskip
\noindent\textbf{Example~2 (Single-step convergence).}
Let $A=\bigl(\begin{smallmatrix}1&1\\1&-1\end{smallmatrix}\bigr)$,
$\bb=\bigl(\begin{smallmatrix}2\\0\end{smallmatrix}\bigr)$,
$\bxi=(1,1)^\top$.
Since $\ip{\ba_1}{\ba_2}=1-1=0$, Corollary~\ref{cor:one_step}
guarantees $\bx^{(1)}=\bxi$ for any $\bx^{(0)}$.
Indeed $A^\top DA=I$, hence $M=0$,
and starting from $\bx^{(0)}=(3,-1)^\top$ gives
$\bx^{(1)}=(1,1)^\top$ exactly.

\section{Conclusions}

We have established three results that sharpen the analysis
of Cimmino's algorithm in~\cite{Guida2023}:
\begin{enumerate}[(i)]
  \item \emph{Exact rate:}
        $\sprad(M_w)$ is the precise per-step contraction factor
        (Theorem~\ref{thm:convergence}, Figure~\ref{fig:envelopes}).
  \item \emph{Closed-form formula for $n=2$:}
        $\varrho=|1-\mu|+\tfrac{1}{2}\sqrt{(w_1-w_2)^2+4w_1w_2\cos^2\theta}$
        shows that the inter-normal angle $\theta$ is the sole
        geometric quantity governing convergence
        (Theorem~\ref{thm:rate}, Figures~\ref{fig:reflection}
        and~\ref{fig:rho_theta}).
  \item \emph{Global optimality of unit weights:}
        $w_1=w_2=1$ uniquely minimises $\sprad(M_w)$,
        attaining $\varrho^*=|\cos\theta|$
        (Theorem~\ref{thm:optimal}, Figure~\ref{fig:rho_theta}).
\end{enumerate}
At $\theta=\pi/2$ the algorithm converges in a single step
(Corollary~\ref{cor:one_step}); the tight-frame condition of
Corollary~\ref{cor:tight_frame} extends this one-step regime to
general~$n$.
The inter-normal angle $\theta$ emerges as a cheap preprocessing
diagnostic: if $|\cos\theta|$ is close to~$1$, Cimmino's method will
converge slowly regardless of weight choice, and a different solver
should be considered.

\section*{Acknowledgements}
The author thank the Tungal School of Basic and Applied Sciences, Jamkhandi, for infrastructure support.

\section*{Competing Interests}
The author declares that there are no competing interests.

\section*{Funding}
No specific funding was received for this work.

\bibliographystyle{elsarticle-num}

\end{document}